\documentclass[conference, twocolumn, 10pt]{IEEEtran}
\usepackage{amsfonts,amsmath,amssymb,amsthm}
\IEEEoverridecommandlockouts
\usepackage{algorithmic}
\usepackage{algorithm}
\usepackage{graphicx}
\usepackage{comment}
\usepackage{cases}
\graphicspath{{./figures_eps/}}
\usepackage{ascmac} 
\usepackage{cite}
\usepackage{subfigure}

\newtheorem{mylem}{Lemma}
\newtheorem{mythm}{Theorem}
\newtheorem{myas}{Assumption}
\newtheorem{myrem}{Remark}

\newcommand{\rfig}[1]{Fig.\,\,\ref{#1}} 
 
\newcommand{\req}[1]{\eqref{#1}}

\newcommand{\rlem}[1]{Lemma\,\ref{#1}}
\newcommand{\ras}[1]{Assumption\,\ref{#1}}

\newcommand{\inte}{\mathbb{N}_{\geq 0}}
\begin{document}
\title{\LARGE \textbf{Self-triggered Model Predictive Control for Continuous-Time Systems: A Multiple Discretizations Approach}}
\author{Kazumune Hashimoto, Shuichi Adachi, and Dimos V. Dimarogonas
\thanks{Kazumune Hashimoto and Shuichi Adachi are with Department of Applied Physics and Physico-Informatics, Keio University, Yokohama, Japan.}
\thanks{Dimos V. Dimarogonas is with the ACCESS Linnaeus Center,
School of Electrical Engineering, KTH Royal Institute of Technology,  Stockholm, Sweden. His work was supported by the Swedish Research Council (VR).
}}
\maketitle
\begin{abstract}
In this paper, we propose a new self-triggered formulation of Model Predictive Control for continuous-time linear networked control systems. Our control approach, which aims at reducing the number of transmitting control samples to the plant, is derived by parallelly solving optimal control problems with different sampling time intervals. The controller then picks up one sampling pattern as a transmission decision, such that a reduction of communication load and the stability will be obtained. 
The proposed strategy is illustrated through comparative simulation examples. 
\end{abstract}
\section{Introduction}
Event-triggered and self-triggered control have been active areas of research in the community of Networked Control Systems (NCSs), due to their potential advantages over the typical time-triggered controllers \cite{heemels2012a,dimos2010a,heemels2013a,heemels2013b,lemmon2009a}. 
In contrast to the time-triggered case where the control signals are executed periodically, event-triggered and self-triggered strategies require the executions based on the violation of prescribed control performances, such as Input-to-State Stability (ISS) \cite{dimos2010a}, LMI based stability conditions \cite{heemels2013b}, and ${\cal L}_2$ gain stability \cite{lemmon2009a}. 
The main difference between these two strategies is that in the event-triggered case an intelligent sensor is required to determine the execution by continuously monitoring the state, while in the self-triggered case the next execution is pre-determined without needing to measure the state continuously.

In another line of research, Model Predictive Control (MPC) has been one of the most successful control strategies applied in a wide variety of applications, such as process industries, robotics, autonomous systems, and moreover, recent research also includes NCSs \cite{findeisen2009a}. 
The basic idea of MPC is to obtain the current control action by solving the Optimal Control Problem (OCP) online, based on the knowledge about current state and predictive behaviors of the dynamics. 

The application of the event-triggered or self-triggered framework to MPC is in particular of importance as the possible way to alleviate communication resources for NCSs. Combining these strategies has been a relatively recent research topic, where most results have been proposed for discrete-time systems see, e.g.,  \cite{evmpc_linear1,evmpc_linear2,evmpc_linear5,evmpc_nonlinear3,evmpc_nonlinear4,evmpc_linear4}, while some results have been proposed for continuous-time systems, e.g.,  \cite{evmpc_nonlinear2,evmpc_nonlinear1,hashimoto2015a,hashimoto2017a}. 
In this paper, we propose a new self-triggered MPC for continuous-time linear systems, as an alternative approach to the preliminary works in the literature. 
In \cite{hashimoto2017a}, the self-triggered condition was derived for continuous-time systems, based on the condition that the optimal cost as a Lyapunov function is guaranteed to decrease. 
Since this result considered input-affine systems, it is applicable to the linear case. 
However, the self-triggered strategy may lead to a conservative result in the following sense; the obtained self-triggered condition includes several parameters, such as Lipschitz constant of stage and terminal costs, which are characterized by the maximum size of state regions. Depending on the problem formulation, therefore, these parameters are sometimes over-approximated and the corresponding self-triggered condition may then become conservative. A related work is also reported in \cite{evmpc_nonlinear2}, where the authors proposed an event-triggered scheme for continuous-time systems. In their approach, the OCP is solved only when the error between the actual and predictive state exceeds a certain thereshold. 

The self-triggered strategy proposed in this paper takes a different problem formulation from previous works in the literature. The basic idea is to parallelly solve OCPs, which provides different transmission time intervals under a piece-wise constant control policy. Based on the different solutions, the controller then selects one solution providing the largest transmission time interval while at the same time guaranteeing the control performance. 
The new formulation of the proposed self-triggered strategy leads to the following main contributions of this paper with respect to the earlier approaches: 
\begin{enumerate}
\item The proposed self-triggered strategy does not include parameters (such as Lipschitz constant parameters) that may be the potential source of conservativeness. The simulation result also illustrates that less conservative results can be obtained than the previous framework. 
\item The optimal costs can be compared under various transmission time intervals. This allows us to obtain the desired control performance by evaluating how much this becomes better or worse according to different values of transmission time intervals. 
\end{enumerate}
This paper is organized as follows. In Section II, the OCP is formulated. In Section III, the self-triggered strategy is provided. In Section IV, the feasibility of our proposed algorithm and the stability are investigated. 
In Section V, the proposed scheme is validated through a numerical example. We finally conclude in Section VI. 

The notations used in the sequel are as follows:
 $\mathbb{R}$,  $\mathbb{R}_{\geq 0}$,  $\mathbb{N}_{\geq 0}$, $\mathbb{N}_{\geq 1}$ are the real, non-negative real, non-negative integers and positive integers, respectively. For a matrix $Q$, we use $Q \succ 0$ to denote that $Q$ is positive definite. Denote $||x||$ as the Euclidean norm of vector $x$. 

\section{Problem Formulation}
\subsection{Dynamics and Cost}
We consider a networked control system depicted in \rfig{network}. 
The dynamics of the plant are assumed to be given by the following linear continuous-time invariant system:
\begin{equation}\label{sys1}
\dot{x} (t) = A x(t) + B u(t)
\end{equation}
where ${x}(t) \in {\mathbb{R}} ^{n}$ is the state and $u(t) \in {\mathbb{R}} ^{m}$ is the control variable. We assume that the pair $(A,B)$ is stabilizable, and $u(t)$ is subject to the constraint $u(t) \in {\cal U}$, where  ${\cal U} \subset {\mathbb{R}} ^{m}$ is a compact subset containing the origin in the interior. The control objective for the MPC is to drive the state to the origin, i.e., $x(t) \rightarrow 0$, as $t\rightarrow \infty$. 

\begin{figure}[tbp]
  \begin{center}
   \includegraphics[width=7.5cm]{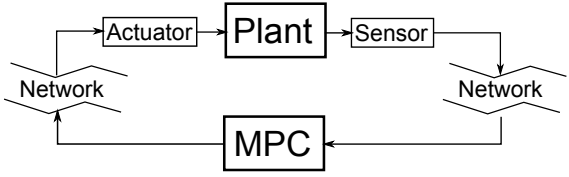}
   \caption{Networked Control System}
   \label{network}
  \end{center}
 \end{figure}
Let $t_k$, $k\in \inte $ be the time instants when OCPs are solved; at $t_k$, the controller solves an OCP based on the knowledge about the state $x(t_k)$, and the dynamics given by \req{sys1}. 
In this paper, we consider the following cost function to be minimized:
\begin{equation}\label{costfunc}
\begin{array}{lll}
J(x(t_k), u(\cdot )) \\
\ \ \ \ \ \ ={\displaystyle \int}^{t_k+T_p} _{t_k}\!\!\! x^{\mathsf{T}} (\xi )Q x(\xi )+ u^{\mathsf{T}} (\xi ) R u (\xi ) {\rm d}\xi \\
\ \ \ \ \ \ \ \ \ \ \ \ \ \ \ \ \ + x^{\mathsf{T}} (t_k + T_p )P_f x(t_k + T_p)
\end{array}
\end{equation}
where $Q\succ 0 $, $R\succ 0$ are the matrices for the stage cost, $P^\mathsf{T} _f = P_f \succ 0$ is the matrix for the terminal cost, and $T_p$ is the prediction horizon. More detailed characterization of $P_f$ will be discussed in later sections. 

In order to derive a self-triggered strategy, we first consider that the control input $u(\xi )$, $\xi \in[t_k, t_k +T_p]$ is constrained to be piece-wise constant with different sampling intervals, e.g., $\delta_1, \delta_2, \cdots, \delta_N$, as shown in \rfig{piecewise}. 
\begin{figure}[tbp]
  \begin{center}
   \includegraphics[width=8.0cm]{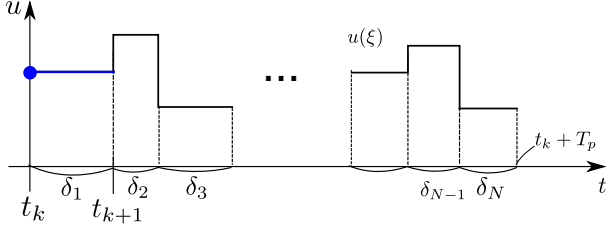}
   \caption{The piece-wise constant control policy considered in this paper. The controller solves an OCP under above piece-wise constant control policy. Once the OCP is solved, the controller transmits the optimal control sample at $t_k$ to the plant (Blue circle). The plant then applies it as sample-and-hold fashion until the next transmission time $t_{k+1}=t_k + \delta_1$ (Blue line), and sends back the new state measurement to the controller. 
   }
   \label{piecewise}
  \end{center}
 \end{figure}
This discretizing scheme is motivated as follows: 
The solution of the OCP by minimizing the cost \req{costfunc} is in general given by a continuous trajectory of the optimal control input, say $u^*(\xi)$, for all $\xi\in [t_k , t_k +T_p]$. 
If the optimal control input \textit{could} be applied until $t_{k+1}$, i.e., $u^* (\xi)$, $\xi \in [t_k , t_{k+1})$, then we could utilize the classic MPC result to guarantee the asymptotic stability of the origin, see \cite{Chen1998a}. However, applying the continuous trajectory of the control input is not suited for practical NCSs applications in terms of the two aspects. 
Firstly, transmitting continuous control trajectory over the network requires an infinite-transmission bandwidth, which is un-realizable. 
Secondly, implementing the {exact} continuous control input is difficult for embedded control system architectures, since they only deal with samples as a discrete time domain, resulting in applying the control input eventually as a sampled-and-hold implementation at a high frequency. As the actual control trajectory for this case possibly differs from the optimal control trajectory, it fails to guarantee the asymptotic stability of the origin.  

The OCP under the piece-wise constant control policy considered in this paper thus provides the optimal control sequence at discrete sampling intervals, i.e., $\{u^* (t_k), u^* (t_k + \delta_1), \cdots, u^* (t_k + \sum^{N} _{j=1} \delta_j ) \}$ rather than the whole control trajectory $u^* (\xi), \xi \in [t_k , t_k + T_p]$. As the procedure of transmitting control samples, we consider the following steps; (i) the controller transmits the optimal control sample $u^* (t_k)$ to the plant; (ii) the plant then applies $u^* (t_k)$ at constant until $t_{k+1} = t_k + \delta_1$; (iii) the plant sends back a new state measurement $x(t_{k+1})$ to the controller to solve the next OCP at $t_{k+1}$. 
Under this procedure, the transmission time interval is then given by $t_{k+1} - t_k = \delta _1$. 

Applying the above transmission procedure not only allows the controller to transmit control command as a sample, but also allows us to formulate the OCP as the discrete time domain. The main difference of the problem formulation with respect to the periodic (or event-triggered) MPC for general discrete time systems is, however, that we are now free to select the sampling time intervals $\delta_1,\cdots, \delta_N$ in an appropriate way.
Although there is a flexibility to select $\delta_1,\cdots, \delta_N$, these intervals must be carefully determined such that:
\begin{enumerate}
\item The asymptotic stability of the origin is guaranteed under MPC with the piece-wise constant control policy. 
\item The reduction of communication load is achieved through the self-triggered formulation. 
\end{enumerate}
In the next subsection, we provide one possible way to 
determine the sampling time intervals $\delta_1,\cdots, \delta_N$, such that the above problems can be tackled. 

\subsection{Determining sampling time intervals}
Under the piece-wise constant control policy outlined in \rfig{piecewise}, the sampling time intervals are determined in this subsection. 
By making use of the flexibility of selecting the sampling time intervals, consider at first that we have \textit{multiple patterns} of sampling time intervals, i.e., we have $M$ ($M \in \mathbb{N}_{\geq 1}$) different sampling patterns in total, where each $i$-th ($i\in \{ 1, 2, \cdots, M \})$ sampling pattern has $N_i$ sampling intervals, $\delta^{(i)} _1 , \delta^{(i)} _2, \cdots, \delta^{(i)} _{N_i}$. 
More specifically, in this paper we consider the sampling patterns as shown in \rfig{controlpatterns}. 
Stated formally, for given $M, N_p \in \mathbb{N}_{\geq 1}$, where $M<N_p$ and $N_p$ represents the maximum number of sampling intervals among all patterns, and $\delta = T_p / N_p$, the sampling time intervals for the $i$-th ($i\in \{ 1,2, \cdots, M \}$) pattern are given by
\begin{equation}\label{sampling_interval}
\delta^{(i)} _1 = i \delta, \ \ \delta^{(i)} _j = \delta\ \ (j = 2, 3, \cdots, N_i ),
 \end{equation} 
with $N_i = N_p -i+1$.
That is, the 1$^{\rm st}$ pattern has the same interval: $\delta^{(1)} _1 = \cdots = \delta^{(1)} _{N_p} = \delta$. The 2$^{\rm nd}$ pattern is the same as the 1$^{\rm st}$ pattern only except the first sampling interval: $\delta^{(2)}_1 = 2 \delta$, $\delta^{(2)} _2 = \cdots = \delta^{(2)} _{N_p -1} = \delta$. Similarly, for the general $i$-th pattern we have $\delta^{(i)} _1 = i \delta$, and $\delta$ for the remaining intervals. 
The controller solves the corresponding OCPs under all sampling patterns above, and then selects one sampling pattern according to the self-triggered strategy proposed in the next section. 

The main motivation of using the sampling patterns shown in \rfig{controlpatterns}, is that it allows to evaluate the trade-off between the transmission interval and the control performance quantitatively. 
According to the transmission procedure given in the previous subsection, the transmission time interval is given by $\delta^{(i)} _1=i\delta$. Thus, using larger patterns leads to longer transmission intervals. From the self-triggered point of view, it is desirable to have larger patterns. However, as we will see in the analysis that follows, the control performance instead becomes worse; this will be proved by the fact that the optimal cost becomes larger as larger patterns are selected. 
In later sections, we will provide a framework of selecting one sampling pattern, such that the trade-off between the transmission time interval and the control performance can be taken into account.
\begin{myrem}[On the selection of $M$]
\normalfont
If the number of patterns $M$ is chosen larger, then we may increase the possibility to have longer transmission time intervals. However, since this leads to the requirement of solving larger number of OCPs at the same time, it will also induce time delays for obtaining optimal solutions under all sampling patterns. 
We will note in later remarks that existing delay compensation strategies may serve as a solution to this problem.
\end{myrem}
\begin{figure}[tbp]
  \begin{center}
   \includegraphics[width=8.7cm]{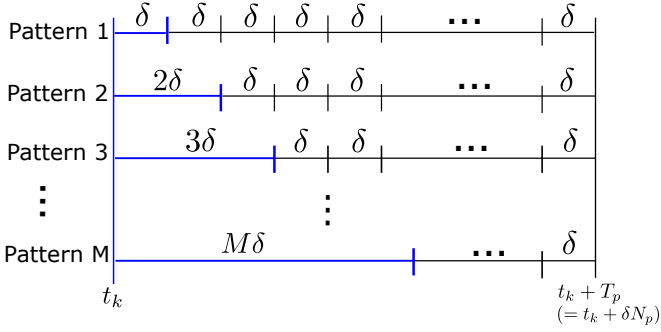}
   \caption{Sampling patterns considered in this paper. Blue lines represent the transmission time intervals. 
   }
   \label{controlpatterns}
  \end{center}
 \end{figure}

\subsection{Optimal Control Problem}
In this subsection the OCP under each sampling pattern is formulated. 
For the $i$-th sampling pattern, we denote 
\begin{equation}
\begin{aligned}
{\bf u}_i (t_k) = \{ u_i (t_k), u_i (t_k + i \delta), u_i (t_k + (i+1) \delta) , \cdots \\ 
\cdots , u_i (t_k +(N_p-1)\delta) \}
\end{aligned}
\end{equation}
as the control input sequence to be applied. 
Note that $u_i (t_k + i \delta)$ is used after $u_i (t_k)$, as $u_i (t_k)$ is applied for the time interval $i\delta$. 
The cost given by \req{costfunc} under the $i$-th sampling pattern can be re-written as 
\begin{equation*}
\begin{array}{lll}
J_i (x(t_k), {\bf u}_i (t_k) ) \\
\ \ \ \ \ = {\displaystyle \int}^{i \delta }_{0}\!\! \left \{ x^\mathsf{T}(t_k + \xi ) Q x(t_k+ \xi ) + u^\mathsf{T} _i (t_k) R u_i (t_k) \right \} {\rm d}\xi \\
\ \ \ \ \ \ \ +  {\displaystyle \sum^{N_p -1} _{n =i }} {\displaystyle \int}^{\delta }_{0}\!\! \left \{ x^\mathsf{T}(t_k+n \delta + \xi ) Q x(t_k+n\delta + \xi ) \right. \\ 
\ \ \ \ \ \ \ \ \ \ \ \ \ \ \ \ \ \ \ \ \ \ \ \ \ \ \ \ \left. + u^\mathsf{T} _i (t_k+ n \delta) R u_i (t_k+n\delta) \right \} {\rm d}\xi \\ 
\ \ \ \ \ \ \ + x^\mathsf{T} (t_k +N_p\delta ) P_f x(t_k + N_p\delta ),
\end{array}
\end{equation*}
where the total cost is separated by each component of the control sequence ${\bf u}_i (t_k)$. 
Here we denoted $J_i$ instead of $J$ to emphasize that the piece-wise constant control policy under the $i$-th sampling pattern is used. By computing each integral in the above equation, 
the total cost for the $i$-th sampling pattern can be translated into a summation of costs: 
\begin{equation*}
\begin{aligned}
J_i (x(t_k), {\bf u}_i (t_k) ) & = F (x(t_k), u_i (t_k), i \delta) \\
& + {\displaystyle \sum^{N_p -1} _{n=i }} \left \{ F ( x(t_k+ n \delta), u_i(t_k+ n \delta), \delta) \right \}  \\
& + x^\mathsf{T}(t_k + N_p \delta) P_f x(t_k + N_p \delta),
\end{aligned}
\end{equation*}
where $F (x(t), u(t), i\delta)$ denotes a new stage cost given by
\begin{eqnarray}\label{costF}
F(x(t),u(t),i \delta) \!\!\! &\!\!\!=&\!\!\!\!\!\!\! \int^{i\delta} _{0}\!\!\!\! x^\mathsf{T} (t + \xi ) Q x(t+ \xi )\! +\! u^\mathsf{T} (t) R u (t) {\rm d}\xi \notag \\
                   &=& \tilde{x}^\mathsf{T}(t) \Gamma (i \delta) \tilde{x}(t),
\end{eqnarray}
where $\tilde{x}(t) = [x^\mathsf{T}(t) \ \ u^\mathsf{T}(t) ]^\mathsf{T}$ and 
\begin{equation*}
\Gamma (i \delta) = \left [
\begin{array}{cc}
{\int}^{i \delta} _{0} A^\mathsf{T}_s Q A_s {\rm d}s &  {\int}^{i \delta} _{0} B^\mathsf{T}_s Q A_s {\rm d}s \\
{\int}^{i \delta} _{0} A^\mathsf{T}_s Q B_s {\rm d}s & {\int}^{i \delta} _{0} (B^\mathsf{T}_s Q B_s + R) {\rm d}s
\end{array}
\right ]
\end{equation*} 
with $A_s = e^{A s}$, $B_s = \int^{s} _{0} e^{A \tau }{\rm d}\tau B$. The OCP for the $i$-th sampling pattern is now formulated as follows.　\\

\noindent
\textit{(Problem 1)} : Given $x(t_k)$, the OCP at $t_k$ for the $i$-th pattern is to minimize $J_i (x(t_k), {\bf u}_i (t_k))$, subject to 
\begin{numcases}
   { }
      {x}(t_k +i \delta ) = A_{i \delta} x(t_k) + B_{i \delta} u_i (t_k) \label{constraint1} \\
   {x}(t_k +(n+1)\delta)\ \  (n=i , \cdots, N_p -1) \notag \\ 
   \ \ \ \ \ \ = A_{\delta} x(t_k +n\delta) + B_{\delta} u_i (t_k + n\delta ) \label{constraint2} \\
    u_i (t_k + n\delta ) \in {\cal U}, \ \ n=0 , i, i+1, \cdots, N_p -1 \label{constraint3} \\
    x (t_k + N_p \delta) \in {\Phi} \label{constraint4}
\end{numcases}
The constraints \req{constraint1} and \req{constraint2} represent the dynamics by applying the control sequence ${\bf u}_i (t_k)$, and \req{constraint3} represents the constraint for the control input. The last constraint \req{constraint4} represents the terminal state penalty, where ${\Phi} = \{ x  \in \mathbb{R}^{n} :  x^\mathsf{T} P_f x\leq \varepsilon \}$ for a given $\varepsilon>0$. 
We let 
\begin{equation*}
\begin{array}{lll}
{\bf u}^* _{i} (t_k) = \{ {u}^* _i(t_k), {u}^* _i(t_k+ i\delta) , \cdots, {u}^* _i (t_k+(N_p -1)\delta) \} \\
{\bf x}^* _{i} (t_k) = \{ {x}^* _i(t_k), {x}^* _i(t_k+ i\delta) , \cdots, {x}^* _i (t_k+N_p \delta) \}
\end{array}
\end{equation*}
be the optimal control and the corresponding state sequence with $x^* _i (t_k) = x(t_k)$, obtained by solving Problem 1. 
We further denote $J^* _i (x(t_k) ) = J_i (x(t_k), {\bf u}^* _{i} (t_k))$ as the optimal cost. 

Similarly to the classic strategy of MPC, 
we consider that the matrix $P_f$ and $\varepsilon$ are chosen such that the following condition on the terminal region $\Phi$ is satisfied: 
\begin{myas}\label{terminal}
There exists a local state feed-back controller $\kappa(x) =K x \in {\cal U}$, satisfying 
\begin{equation}\label{localcontroller}
\begin{aligned}
x(t_{k}+\delta)^\mathsf{T} P_f x(t_{k}+\delta) &- x^\mathsf{T}(t_k) P_f x(t_k) \\ 
& \leq - F (x(t_k), K x(t_k), \delta )
\end{aligned}
\end{equation}
for all $x(t_k) \in \Phi$, where $x(t_{k}+\delta) = (A_\delta +B_\delta K) x(t_k)$. 
\end{myas}
\ras{terminal} will be used to guarantee that the optimal cost decreases along the time by an appropriate selection of the sampling pattern. 
Since the system \req{sys1} is assumed to be stabilizable, the local controller $\kappa(x)$ and $\Phi$ satisfying \req{localcontroller}, can be found off-line  by following the procedure presented in \cite{Chen1998a}. 
To arrive at the self-triggered strategy, we will in the following derive some useful properties for the optimal costs obtained under different sampling patterns. These properties are key ingredients to quantify the control performances for the self-triggered strategy, as well as for the asymptotic stability provided in later sections. 
\begin{mylem}\label{lem1}
Suppose that Problem 1 admits a solution at $t_k$ under each sampling pattern $i\in \{ 1, 2, \cdots, M \}$, which provides the optimal costs $J^* _i (x(t_k))$ for all $i\in \{1,\cdots,M\}$. Then we have
\begin{equation}\label{optimalcostincrease}
J^* _1 (x(t_k)) \leq J^* _2 (x(t_k)) \leq J^* _3 (x(t_k)) \cdots \leq J^* _M (x(t_k))
\end{equation}
\end{mylem}
\begin{proof}
Let ${\bf u}^* _{i} (t_k)$, ${\bf x}^* _{i} (t_k)$, $i\in \{1,2, \cdots, M\}$ be the optimal control and the corresponding state sequence obtained by Problem 1 under the $i$-th sampling pattern. The illustration of the corresponding optimal piece-wise constant control policy is depicted in \rfig{piecewisepatterni}. 
\begin{figure}[tbp]
  \begin{center}
   \includegraphics[width=8.7cm]{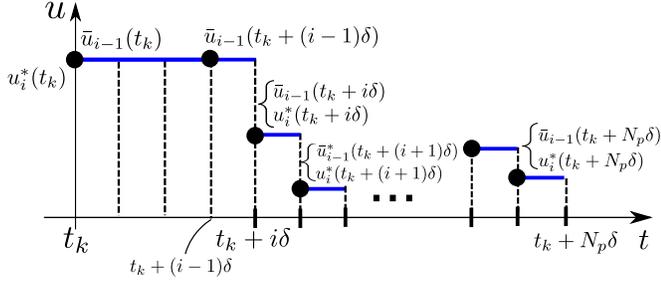}
   \caption{Optimal piece-wise constant control policy for the $i$-th sampling pattern (blue line). The control policy in the figure also provides a feasible solution to Problem 1 under the ($i-1$)-th sampling pattern, since $u^* _i (t_k)$ is applied for $t\in [t_k, t_k + (i-1)\delta ) \in [t_k, t_k+i\delta)$. The black circles represent the admissible control sequence for the ($i-1$)-th pattern ${\bf \bar{u}}_{i-1} (t_k)$.}
   \label{piecewisepatterni}
  \end{center}
 \end{figure}
Under the $i$-th ($i\geq 2$) sampling pattern, $u^* _i (t_k)$ is applied at constant for all $t\in [t_k, t_k + i\delta )$ as shown in \rfig{piecewisepatterni}. 
The control policy for the $i$-th ($i \geq 2$) sampling pattern is thus {admissible} also for the ($i-1$)-th sampling pattern, 
as $u^* _i (t_k)$ is applied for $t\in [t_k, t_k + (i-1)\delta ) \in [t_k, t_k+i\delta)$.
More specifically, let 
\begin{equation*}
\begin{aligned}
{\bf \bar{u}} _{i-1} (t_k) = \{ \bar{u} _{i-1} (t_k), \bar{u} _{i-1} (t_k+(i-1) \delta) \cdots \ \ \\
\ \ \ \cdots, \bar{u} _{i-1} (t_k+(N_p -1)\delta) \},
\end{aligned}
\end{equation*}
where $\bar{u}_{i-1} (t_k) = u^* _i (t_k)$, $\bar{u}_{i-1} (t_k+(i-1) \delta) = u^* _i (t_k)$ and 
\begin{equation*}
\bar{u} _{i-1} (t_k+ j \delta) = u^* _{i} (t_k+ j \delta ), \ \ j= i, \cdots, N_p -1, 
\end{equation*}
and ${\bf \bar{x}} _{i-1} (t_k) = \{ \bar{x} _{i-1} (t_k), \bar{x} _{i-1} (t_k+(i-1) \delta) \cdots, \bar{x} _{i-1} (t_k+N_p\delta) \}$ be the corresponding state sequence with $\bar{x} _{i-1} (t_k) = x(t_k)$ (see the illustration of ${\bf \bar{u}} _{i-1}$ in \rfig{piecewisepatterni}). 
Then, ${\bf \bar{u}} _{i-1} (t_k)$ provides a feasible solution to Problem 1 under the $(i-1)$-th pattern, satisfying all constraints \req{constraint1}, \req{constraint2}, \req{constraint3} and \req{constraint4}. 
The last constraint \req{constraint4} is obtained by the fact that $\bar{x} _{i-1} (t_k+ N_p \delta) = {x}^* _{i} (t_k+ N_p\delta) \in \Phi$.
Since ${\bf \bar{u}} _{i-1}$ is a feasible controller for the $(i-1)$-th pattern, we obtain
\begin{equation}
\begin{array}{lll}
J^* _{i-1} (x(t_k)) &\leq & J _{i-1} (x(t_k), {\bf \bar{u}}_{i-1} (t_k) ) \\
                        & = & J _{i} (x(t_k), {\bf u}^* _{i} (t_k) ) \\
                        & = & J^* _{i} (x(t_k)),
\end{array}
\end{equation}
and the above inequality holds for all $i\in \{2, 3, \cdots, M\}$.
The proof is thus complete. 
\end{proof}
\rlem{lem1} states that the 1$^{\rm st}$ pattern provides the best control performance in the sense that the optimal cost takes the minimum value among all patterns, and moreover, the control performance becomes worse as larger patterns are selected.
The next lemma states that the optimal cost is guaranteed to decrease whenever the 1$^{\rm st}$ pattern is used: 
\begin{mylem}
\label{lem2}
Suppose that the $i$-th pattern was used at $t_{k-1}$ and the next time to solve the OCP is given by $t_k =t_{k-1} + i \delta$. 
Then, under \ras{terminal}, the optimal cost satisfies
\begin{equation}\label{stability}
J^* _1 (x(t_k)) -  J^* _i (x(t_{k-1})) \leq - F (x(t_{k-1}), u^* _i (t_{k-1}), i\delta)
\end{equation}
\end{mylem}
\begin{proof}
(Sketch) Let 
\begin{equation*}
\begin{aligned}
{\bf u}^* _{i} (t_{k-1}) &= \{ {u}^* _i(t_{k-1}), {u}^* _i(t_k) , \cdots, {u}^* _i (t_k+(N_p -i-1)\delta) \} \\
{\bf x}^* _{i} (t_{k-1}) &= \{ {x}^* _i(t_{k-1}), {x}^* _i(t_k) , \cdots, {x}^* _i (t_k+(N_p -i)\delta) \}
\end{aligned}
\end{equation*}
be the optimal control input and the corresponding state sequence obtained at $t_{k-1}$ under the $i$-th pattern. From the constraint \req{constraint4}, we have ${x}^* _i (t_k+(N_p -i)\delta ) \in \Phi$. At $t_k$, we consider the following control and the corresponding state sequence for the 1$^{\rm st}$ pattern; ${\bf \bar{u}} _{1} (t_{k}) = \{ \bar{u} _1(t_{k}), \bar{u} _1(t_k+\delta) , \cdots, \bar{u} _1 (t_k+(N_p -1)\delta) \}$, ${\bf \bar{x}} _{1} (t_{k}) = \{ \bar{x} _1(t_{k}), \bar{x} _1(t_k+\delta) , \cdots, \bar{x} _1 (t_k+N_p\delta) \}$, where each component of ${\bf \bar{u}} _{1} (t_{k})$ is given by
\begin{equation}\label{utilde}
\begin{array}{lll}
\bar{u}_1 (t_k +j \delta) = 
\left \{
\begin{array}{l}
{u}^* _i(t_k +j\delta)\\
\ \ \ ({\rm for}\ j=0,\cdots, N_p -i-1 ) \\
\kappa (\bar{x}_1(t_k +j\delta))\\
\ \ \ ({\rm for}\ j=N-i , \cdots, N_p -1)
\end{array}
\right. 
\end{array}
\end{equation}
Applying the local controller $\kappa$ from $t_k +(N_p -i)\delta$ is admissible since we have $\bar{x}_1 (t_k+(N_p-i)\delta) = {x}^* _i (t_k+(N_p-i)\delta) \in \Phi$. Thus ${\bf \bar{u}} _{1} (t_{k})$ is a feasible controller for Problem 1 under the 1$^{\rm st}$ sampling pattern, 
and the upper bound of the difference between $J^* _1 (x(t_{k}))$ and $J^* _i (x(t_{k-1}))$ is given by
\begin{equation}\label{costdif}
\begin{aligned}
J^* _1 (x(t_{k})) - J^* _i (x(t_{k-1})) &\leq J_1 (x(t_{k}), {\bf \bar{u}} _{1} (t_{k})) \\
                                                &\ \  - J_i (x(t_{k-1}), {\bf u}^* _{i} (t_{k-1}))
\end{aligned}
\end{equation}
Some calculations of the right hand side in \req{costdif} yield \req{stability}. 
The derivation of \req{stability} from \req{costdif} is given in the Appendix. 
\end{proof}
\section{Self-triggered strategy}
In this section we propose the self-triggered strategy as one of our main results. The key idea of the framework
is to select the {best} pattern in the sense that it provides the largest possible transmission time interval, while satisfying some conditions to obtain the desired control performance. In the following proposed algorithm, we denote $i_k$, $k\in \inte$ as the sampling pattern selected by the controller to transmit the corresponding optimal control sample $u^* _{i_k} (t_k)$.\\ 

\noindent
{\bf \textit{Algorithm 1}}: {\bf {(Self-triggered MPC strategy)}}
\begin{enumerate}
\item \textit{Initialization :} At initial time $t_0$, the controller solves Problem 1 only for the 1$^{\rm st}$ sampling pattern based on $x(t_0)$. The controller then transmits the optimal control sample $u^* _1(t_0)$ to the plant, i.e., $i_0 = 1$. The plant applies the constant controller $u^* _1(t_0)$ until $t_1=t_0 +\delta$, and sends back $x(t_1)$ to the controller as a new state measurement. 
\item At $t_k$, $k\in \mathbb{N}_{\geq 1}$, the controller solves Problem 1 for all patterns $i=1,\cdots, M$ based on $x(t_k)$. This provides the optimal control sequences ${\bf u}^* _1 (t_k)$, ${\bf u}^* _2 (t_k)$, $\cdots$, ${\bf u}^* _M (t_k)$, and the corresponding optimal costs  $J^* _1 (x(t_k)), \cdots, J^* _M (x(t_k))$. 
\item The controller selects one pattern $i _k \in \{ 1, \cdots, M \} $ by solving the following problem; 
\begin{equation}\label{maxi}
i _{k} = \underset{i\in \{1,2, \cdots, M \} } {\rm max}\ i ,
\end{equation}
subject to 
\begin{equation}\label{feasibility1}
J^* _i(x(t_k)) \leq J^* _1 (x(t_k)) + \beta
\end{equation}
\begin{equation}\label{feasibility2}
\begin{aligned}
J^* _{i} (x(t_k)) \leq & J^* _{i _{k-1}} (x(t_{k-1})) \\
                          & - \gamma F (x(t_{k-1}), u^* _{i_{k-1}} (t_k), i_{k-1} ),
\end{aligned}
\end{equation}
where $\beta$ and $\gamma$ are the constant parameters, satisfying $0\leq \beta$, $0<\gamma \leq 1$. 
\item The controller transmits $u^* _{i_k} (t_k)$, and then the plant applies $u^* _{i_k} (t_k)$ as sample-and-hold implementation until $t_{k+1} = t_k+ i_k \delta$. The plant then sends back $x(t_{k+1})$ to the controller as a new current state measurement.
\end{enumerate}

The main point of our proposed algorithm is the way to select the pattern $i_k$ given in the step (iii). 
From \rlem{lem1}, the 1$^{\rm st}$ pattern provides the minimum cost among all sampling patterns. Thus, the first condition \req{feasibility1} implies that larger patterns are allowed to be selected to obtain longer transmission intervals, but the optimal cost should not go far from the 1$^{\rm st}$ pattern; the optimal cost is allowed to be larger only by $\beta$ from $J^* _1(x(t_k))$, so that it does not degrade much the control performance. 
Thus, the parameter $\beta$ plays a role to regulate the trade-off between the control performance and the transmission time intervals. That is, a smaller $\beta$ leads to better control performance (but resulting in less transmissions), and larger $\beta$ leads to less transmissions (but resulting in worse control performance). 

The second condition \req{feasibility2} takes into account the optimal cost obtained at the previous time $t_{k-1}$, and this aims at guaranteeing the asymptotic stability of the origin. Note that $\gamma$ needs to satisfy $0<\gamma \leq 1$. As we will describe in the next section, this condition ensures that Algorithm 1 is always implementable. 
Since it is desirable to reduce the communication load as much as possible, the controller selects the pattern providing the largest transmission interval satisfying \req{feasibility1}, \req{feasibility2}, i.e., max $i$ in \req{maxi}.

The main advantage of using the proposed method is that the optimal cost $J^* _i (t_k)$ can be compared not only with the previous one $J^* _{i_{k-1}} (t_{k-1})$, but also with the current ones obtained at $t_k$ under different sampling patterns. This allows us not only to ensure stability, but also to evaluate how much the control performance becomes better or worse according to the transmission time intervals. Note that the control performance may also be regulated through the tuning of $\gamma$ in \req{feasibility2}. 
However, due to the condition $0<\gamma \leq 1$, we cannot select $\gamma$ large enough such that small patterns (good control performance) are ensured to be obtained. 
Thus the desired control performance can be suitably specified through the first condition \req{feasibility1}, rather than \req{feasibility2}. 

Note also that in contrast to our preliminary work in \cite{hashimoto2017a}, 
Algorithm 1 does not involve parameters such as Lipschitz constants for the stage and the terminal cost. 
Since these parameters involve the maximum distance of the state from the origin, i.e., ${\rm sup}_{t\in[0, \infty)} \{ ||x(t)|| \}$ (see e.g., \textit{Lemma 3.2} in \cite{zhu}), they may need to be over-approximated and the self-triggered condition may then become conservative. We will also illustrate through a simulation example that the proposed method attains a less conservative result than the previous approach. 

\begin{myrem}[{Effect of time delays}]
\normalfont
The main drawback of Algorithm 1 is the requirement of solving multiple OCPs at the same time, which clearly induces a time-delay of transmitting control samples in practical implementations. Regarding time delays, several methods have been proposed to take them into account and can also be applied to our proposed self-triggered strategy. For example, a delay compensation strategy has been proposed in \cite{findeisen2009a}.  
When applying this approach, the maximum total time delay $\bar{\tau}_d$ needs to be upper bounded to satisfy $\delta^{(i)} _1 < T_p - \bar{\tau}_d$ in order to guarantee stability. This implies that the condition $i < (T_p - \bar{\tau}_d)/\delta$ is required in addition to the conditions \req{feasibility1}, \req{feasibility2} as the rule to choose the sampling pattern.
\end{myrem}
\begin{myrem}[Effect of the noise or model uncertainties]
\normalfont
In the above formulation, we have not considered any effects of model uncertainties or disturbances. However, the proposed scheme can be extended to take into account these effects by slightly modifying \rlem{lem2}. 
Suppose that the actual state is given by $\dot{x} = A x + Bu + w$, where $w$ denotes additive uncertainties or disturbances satisfying $||w||\leq w_{\rm max}$. 
By utilizing Theorem 2 in \cite{camacho2002a}, we can show that there exists a positive $L_v$ such that $J^* _1 (x(t_k)) -  J^* _i (x(t_{k-1})) \leq - F (x(t_{k-1}), u^* _i (t_{k-1}), i\delta) + L_v w_{\rm max}$ instead of \req{stability}. Therefore, assuming that $w_{\rm max}$ is known, the corresponding self-triggered strategy is obtained by adding $L_v w_{\rm max}$ to the right hand side of \req{feasibility2}. 
Note that the first condition \req{feasibility1} does not need to be modified, since \rlem{lem1} still holds even for the disturbance case. 
\end{myrem}

\section{Analysis}
One of the desirable properties of Algorithm 1 is to ensure that it is always implementable, i.e., we need to exclude the case when all the patterns do not satisfy both \req{feasibility1} and \req{feasibility2}. Furthermore, the stability of the closed loop system under Algorithm 1 needs to be verified.
In the following theorem, we deduce that both of these properties are satisfied. 
\begin{mythm}
Consider the networked control system in \rfig{network} where the plant follows the dynamics given by \req{sys1} and the proposed self-triggered strategy (Algorithm 1) is implemented.  
The followings are then satisfied: 
\begin{enumerate}
\item The way to obtain the pattern $i_k$ in step (iii) in Algorithm~1, is always feasible. 
That is, there exists at least one pattern ${i}$, satisfying both \req{feasibility1}, \req{feasibility2} for all $k\in \mathbb{N}_{\geq 0}$. 
\item The closed loop system is asymptotically stabilized to the origin. 
\end{enumerate}
\end{mythm}
\begin{proof}
The proof of (i) is obtained by showing that the 1$^{\rm st}$ sampling pattern ($i=1$) always satisfies \req{feasibility1} and \req{feasibility2}. 
The first condition is clearly satisfied when $i=1$ since $\beta \geq 0$. Furthermore, from \rlem{lem2}, we obtain
\begin{equation*}
\begin{aligned}
 J^* _{1} (x(t_k))  & \leq  J^* _{i _{k-1}} (x(t_{k-1})) - F (x(t_{k-1}), u^* _{i_{k-1}} (t_{k-1}), i_{k-1} ) \\ 
                       & \leq J^* _{i _{k-1}} (x(t_{k-1})) -\gamma F (x(t_{k-1}), u^* _{i_{k-1}} (t_{k-1}) , i_{k-1}) 
\end{aligned}
\end{equation*}
Thus the second condition holds for $i=1$. Thus the proof of (i) is complete. 

The proof of (ii) is obtained by the fact that the optimal cost decreases along the time sequence. 
Since the optimal cost of the selected pattern satisfies \req{feasibility2}, we have 
\begin{equation*}
\begin{aligned}
J^* _{i_1} (x(t_1)) - J^* _{i_0} (x(t_{0}))  &\leq - \gamma F (x(t_{0}), u^* _{i_{0}} (t_{0}) , i_{0}) \\
                                                   &< - \gamma \int^{t_1} _{t_0} x^\mathsf{T}(t) Q x(t) {\rm d}t \\
J^* _{i_2} (x(t_2)) - J^* _{i _{1}} (x(t_{1}))  &\leq - \gamma F (x(t_{1}), u^* _{i_{1}} (t_{1}) , i_{1}) \\
                                                     &< - \gamma \int^{t_2} _{t_1} x^\mathsf{T}(t) Q x(t) {\rm d}t\\
                                                     & \vdots   
\end{aligned}
\end{equation*}
where the derivation from the first to the second in-equality follows from the definition of the stage cost $F$ given by \req{costF}.
Summing over both sides of the above yields 
\begin{equation*}
\begin{aligned}
\gamma \int^{\infty} _{t_0} x^\mathsf{T}(t) Q x(t) {\rm d}t < {J^* _{i_0} (x(t_{0})) - J^* _{i_\infty} (x(\infty))} < \infty
\end{aligned}
\end{equation*}
Since the function $x^\mathsf{T}(t) Q x(t)$ is uniformly continuous on $t\in [0, \infty )$ and $Q\succ 0$, 
we obtain $||{x} (t)|| \rightarrow 0$ as $t\rightarrow \infty$ from Barbalat's lemma \cite{khalil}. This completes the proof. 
\end{proof}

\section{Illustrative example}
As an illustrative example, we consider the spring-mass system; the state vector $x=[x_1; x_2]$ consists of the position $x_1$ and the velocity $x_2$, and the dynamics are given by 
\begin{equation}
\dot{x}  = \left [
\begin{array}{cc}
0 &  1 \\
-k/m & 0  \\
\end{array}
\right ]x + \left [
\begin{array}{c}
0 \\
1/m \\
\end{array}
\right ] u, 
\end{equation}
where $k=2$ is the spring coefficient and $m=1$ is the mass. 
The matrices for the stage cost are $Q=I_2$, $R=0.5$, and the prediction horizon is $T_p = 8$. 
The terminal matrix $P_f$ and the local controller $\kappa$ are computed properly by following the procedure presented in \cite{Chen1998a}.
We further assume that the control input $u$ is constrained by $||u|| \leq 8$. 
The total number of sampling patterns is given by $M=30$ with $\delta = 0.1$, i.e., the maximum transmission time interval is 
$M\delta = 3$. 
\begin{figure}[tbp]	
   \centering
   \includegraphics[width=8.0cm]{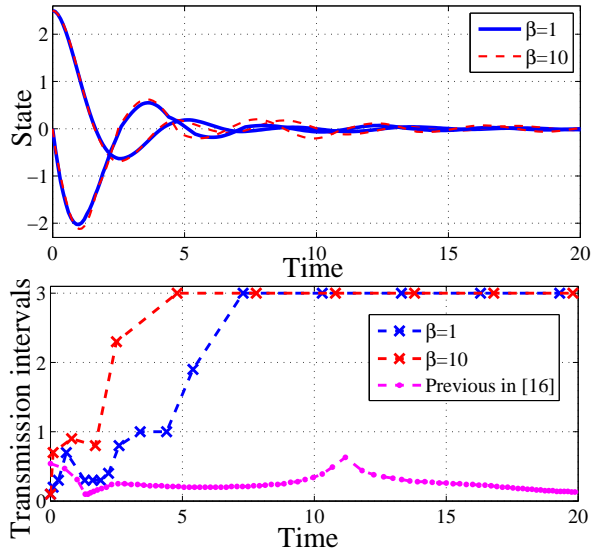}
   \caption{State trajectories and transmission time intervals.}
   \label{state_patterns}
\end{figure}

\rfig{state_patterns} shows state trajectories of $x_1$ and $x_2$ (upper), with $\gamma = 0.5$, $\beta = 1$ and $\beta=10$ from the initial state $x_0 = [2.5;\ 0]$, and the transmission time intervals (lower). 
From the figure, the state achieves asymptotic stability of the origin, and larger patterns (i.e., longer transmission time intervals) are more likely to be obtained as the state gets closer to the origin. One can also see the trade-off between the control performance and the number of transmissions; faster convergence is achieved when $\beta = 1$ than $\beta=10$ from the upper figure, while it requires more transmissions of control samples as shown in the lower figure.

To compare with the previous framework, 
we have also plotted the transmission time intervals in \rfig{state_patterns} obtained by the methodology presented in \cite{hashimoto2017a}. Here we set $\sigma=\gamma=0.5$ in {Eq.(19)} in \cite{hashimoto2017a}, 
to ensure the same rate of cost decrease. 
From \rfig{state_patterns}, the proposed scheme attains much longer transmission time intervals than the previous method under the same performance guarantees. 
\section{Conclusion and Future work}
In this paper, we propose a self-triggered control methodology for continuous-time linear networked control systems. 
Our proposed scheme was derived by solving multiple optimal control problems with different sampling time intervals, and the controller selects one sampling pattern resulting in the largest transmission time intervals while satisfying the desired control performances. Our proposed scheme was also validated by an illustrative example. Future work involves deriving the self-triggered strategies against random packet dropouts and extend the proposed result to the nonlinear case. 

\appendix
{(Derivation of \req{stability} from \req{costdif})}:
The optimal cost for the $i$-th pattern at $t_{k-1}$ is given by 
\begin{equation*}
\begin{aligned}
J_i (x(t_{k-1}), {\bf u}^*_i & (t_{k-1}) ) = F (x(t_{k-1}), u^* _i (t_{k-1}), i \delta) \\
& + F (x^* _i (t_{k}), u^* _i (t_{k}), \delta) \\
& + {\displaystyle \sum^{N_p-i-1} _{n=1}} F ( x^* _i (t_k+ n \delta), u^* _i(t_k+ n \delta), \delta)  \\
& + {x^* _i}^\mathsf{T} (t_k + (N_p-i) \delta) P_f x^* _i (t_k + (N_p-i) \delta).
\end{aligned}
\end{equation*}
Furthermore, the cost at $t_k$ under the 1$^{\rm st}$ sampling pattern with 
${\bf \bar{u}}_1 (t_k)$ in \req{utilde}, is given by 
\begin{equation*}
\begin{aligned}
J_1 (x(t_k), {\bf \bar{u}}_1 (t_k) ) & = F (x(t_k), \bar{u}_1 (t_k), \delta) \\
& + {\displaystyle \sum^{N_p -1} _{n=1}} F ( \bar{x}_1(t_k+ n \delta), \bar{u}_1(t_k+ n \delta), \delta) \\
& + {\bar{x}}^\mathsf{T} _1(t_k + N_p \delta) P_f \bar{x}_1(t_k + N_p \delta).
\end{aligned}
\end{equation*}
From \req{utilde}, we have $\bar{u}_1 (t_k +j \delta) = {u}^* _i (t_k +j \delta)$ for $j=0,\cdots, N_p -i-1$, and thus 
$\bar{x}_1 (t_k +j \delta) = {x}^* _i (t_k +j \delta)$ for $j=0,\cdots, N_p -i$. 
The difference between $J_1 (x(t_{k}), {\bf \bar{u}} _{1} (t_{k}))$ and $J_i (x(t_{k-1}), {\bf u}^* _{i} (t_{k-1}))$ is then given by 
\begin{equation*}
\begin{aligned}
J_1 (x(t_{k}), {\bf \bar{u}} _{1} & (t_{k})) - J_i (x(t_{k-1}), {\bf u}^* _{i} (t_{k-1})) \\
                                                     & =  - F (x(t_{k-1}), u^* _i(t_{k-1}), i \delta) \\
                                                      & + {\displaystyle \sum^{N_p-1} _{n=N_p-i }} F ( \bar{x} (t_{k}+ n \delta), \kappa (\bar{x} (t_{k}+ n \delta), \delta) \\ 
&- {\bar{x}}^\mathsf{T} _1 (t_{k} + (N_p-i) \delta) P_f \bar{x} _1 (t_{k} + (N_p-i) \delta) \\ 
&+ {\bar{x} }^\mathsf{T} _1 (t_{k} + N_p\delta) P_f \bar{x}_1 (t_{k} + N_p\delta) 
\end{aligned}
\end{equation*}
From \req{localcontroller}, we have 
${\bar{x} }^\mathsf{T} _1 (t_{k} + N_p\delta) P_f \bar{x}_1 (t_{k} + N_p\delta)
- {\bar{x} }^\mathsf{T} _1 (t_{k} + (N_p-1) \delta) P_f \bar{x}_1 (t_{k} + (N_p-1) \delta)
\leq - F( \bar{x}_1 (t_{k}+ (N_p-1) \delta), \kappa (\cdot ) , \delta)$.
By using this inequality, we obtain
\begin{equation*}
\begin{aligned}
J_1 (x(t_{k}), {\bf \bar{u}} _{1} (t_{k}&))- J_i (x(t_{k-1}), {\bf u}^* _{i} (t_{k-1})) \\
&\leq - F (x(t_{k-1}), u^* _i(t_{k-1}), i\delta) \\ 
&  + {\displaystyle \sum^{N_p-2} _{n=N_p-i }} F( \bar{x}_1 (t_{k}+ n \delta), \kappa (\bar{x}_1 (t_{k}+ n \delta), \delta)\\ 
& - {\bar{x}}^\mathsf{T} _1 (t_{k} + (N_p-i) \delta) P_f \bar{x}_1 (t_{k} + (N_p -i) \delta) \\
& +  {\bar{x} }^\mathsf{T} _1 (t_{k} + (N_p-1) \delta ) P_f \bar{x}_1 (t_{k} + (N_p-1) \delta ) \\
\end{aligned}
\end{equation*}
Similarly above, by recursively using the inequality from \req{localcontroller};
\begin{equation*}
\begin{aligned}
{\bar{x} }^\mathsf{T} _1 & (t_{k} + (N_p-j) \delta) P_f \bar{x}_1 (t_{k} + (N_p-j-1)\delta) \\
                                & - {\bar{x} }^\mathsf{T} _1 (t_{k} + (N_p-j-1) \delta) P_f \bar{x}_1 (t_{k} + (N_p-j-1) \delta) \\
                                & \leq - F( \bar{x}_1 (t_{k}+ (N_p-j-1) \delta), \kappa (\cdot ) , \delta)
\end{aligned}
\end{equation*}
for $j\in \{1, 2, \cdots, i-1\}$, we obtain
\begin{equation*}
\begin{aligned}
J_1 (x(t_{k}), {\bf \bar{u}} _{1} (t_{k})) & - J_i (x(t_{k-1}), {\bf u}^* _{i} (t_{k-1})) \\
 & \leq - F (x(t_{k-1}), u^* _i(t_{k-1}), i\delta),
\end{aligned}
\end{equation*}
and this yields \req{stability}.


\begin{thebibliography}{10}
\providecommand{\url}[1]{#1}
\csname url@samestyle\endcsname
\providecommand{\newblock}{\relax}
\providecommand{\bibinfo}[2]{#2}
\providecommand{\BIBentrySTDinterwordspacing}{\spaceskip=0pt\relax}
\providecommand{\BIBentryALTinterwordstretchfactor}{4}
\providecommand{\BIBentryALTinterwordspacing}{\spaceskip=\fontdimen2\font plus
\BIBentryALTinterwordstretchfactor\fontdimen3\font minus
  \fontdimen4\font\relax}
\providecommand{\BIBforeignlanguage}[2]{{%
\expandafter\ifx\csname l@#1\endcsname\relax
\typeout{** WARNING: IEEEtran.bst: No hyphenation pattern has been}%
\typeout{** loaded for the language `#1'. Using the pattern for}%
\typeout{** the default language instead.}%
\else
\language=\csname l@#1\endcsname
\fi
#2}}
\providecommand{\BIBdecl}{\relax}
\BIBdecl

\bibitem{heemels2012a}
W.~P. M.~H. Heemels, K.~H. Johansson, and P.~Tabuada, ``An introduction to
  event-triggered and self-triggered control,'' in \emph{Proceedings of the
  51st IEEE Conference on Decision and Control (IEEE CDC)}, 2012, pp.
  3270--3285.

\bibitem{dimos2010a}
A.~Eqtami, D.~V. Dimarogonas, and K.~J. Kyriakopoulos, ``Event-triggered
  control for discrete time systems,'' in \emph{Proceedings of American Control
  Conference (ACC)}, 2010, pp. 4719--4724.

\bibitem{heemels2013a}
W.~P. M.~H. Heemels and M.~C.~F. Donkers, ``Model-based periodic
  event-triggered control for linear systems,'' \emph{Automatica}, vol.~49,
  no.~3, pp. 698--711, 2013.

\bibitem{heemels2013b}
W.~P. M.~H. Heemels, M.~C.~F. Donkers, and A.~R. Teel, ``Periodic
  event-triggered control for linear systems,'' \emph{IEEE Transaction on
  Automatic Control}, vol.~58, no.~4, pp. 847--861, 2013.

\bibitem{lemmon2009a}
X.~Wang and M.~D. Lemmon, ``Self-triggered feedback control systems with finite
  $\mathcal{L}_2$ gain stability,'' \emph{IEEE Transaction on Automatic
  Control}, vol.~54, no.~3, pp. 452--467, 2009.

\bibitem{findeisen2009a}
P.~Varutti and R.~Findeisen, ``Compensating network delays and information loss
  by predictive control methods,'' in \emph{Proceedings of European Control
  Conference (ECC)}, 2009, pp. 1722--1727.

\bibitem{evmpc_linear1}
D.~Lehmann, E.~Henriksson, and K.~H. Johansson, ``Event-triggered model
  predictive control of discrete-time linear systems subject to disturbances,''
  in \emph{Proceedings of European Control Conference (ECC)}, Strasbourg,
  France, 2013, pp. 1156--1161.

\bibitem{evmpc_linear2}
J.~D. J.~B. Berglind, T.~M.~P. Gommans, and W.~P. M.~H. Heemels,
  ``Self-triggered mpc for constrainted linear systems and quadratic costs,''
  in \emph{Proceedings of IFAC Nonlinear Model Predictive Control Conference},
  2012, pp. 342--348.

\bibitem{evmpc_linear5}
F.~D. Brunner, W.~P. M.~H. Heemels, and F.~Allgower, ``Robust self-triggered
  mpc for constrained linear systems,'' in \emph{Proceedings of American
  Control Conference (ACC)}, 2014, pp. 472--477.

\bibitem{evmpc_nonlinear3}
A.~Eqtami, S.~Heshmati-Alamdari, D.~V. Dimarogonas, and K.~J. Kyriakopoulos,
  ``A self-triggered model predictive control framework for the cooperation of
  distributed nonholonomic agents,'' in \emph{Proceedings of the 52nd IEEE
  Conference on Decision and Control}, Firenze, Italy, 2013, pp. 7384--7389.

\bibitem{evmpc_nonlinear4}
T.~M.~P. Gommans and W.~P. M.~H. Heemels, ``Resource-aware mpc for constrainted
  nonlinear systems: A self-triggered control approach,'' \emph{Systems Control
  Letters}, pp. 59--67, 2015.

\bibitem{evmpc_linear4}
T.~Gommans, D.~Antunes, T.~Donkers \emph{et~al.}, ``Self-triggered linear
  quadratic control,'' \emph{Automatica}, vol.~50, no.~4, pp. 1279--1287, 2014.

\bibitem{evmpc_nonlinear2}
H.~Li and Y.~Shi, ``Event-triggered robust model predictive control of
  continuous-time nonlinear systems,'' \emph{Automatica}, vol.~50, no.~5, pp.
  1507--1513, 2014.

\bibitem{evmpc_nonlinear1}
A.~Eqtami, S.~Heshmati-Alamdari, D.~V. Dimarogonas, and K.~J. Kyriakopoulos,
  ``Self-triggered model predictive control for nonholonomic systems,'' in
  \emph{Proceedings of European Control Conference (ECC)}, Strasbourg, France,
  2013, pp. 638--643.

\bibitem{hashimoto2015a}
K.~Hashimoto, S.~Adachi, and D.~V. Dimarogonas, ``Time-constrained
  event-triggered model predictive control for nonlinear continuous-time
  systems,'' in \emph{Proceedings of the 54th {IEEE} Conference on Decision and
  Control (IEEE CDC)}, 2015, pp. 4326--4331.

\bibitem{hashimoto2017a}
------, ``Self-triggered model predictive control for nonlinear input-affine
  dynamical systems via adaptive control samples selection,'' \emph{IEEE
  Transactions on Automatic Control}, to appear. Preprint available at
  http://arxiv.org/pdf/1603.03677v1.pdf.

\bibitem{Chen1998a}
H.~Chen and F.~Allgower, ``A quasi-infinite horizon nonlinear model predictive
  control with guaranteed stability,'' \emph{Automatica}, vol.~34, no.~10, pp.
  1205--1217, 1998.

\bibitem{zhu}
Y.~Zhu and U.~Ozuner, ``Robustness analysis on constrained model predictive
  control for nonholonomic vehicle regulation,'' in \emph{Proceedings of
  American Control Conference (ACC)}, Chicago, USA, 2009, pp. 3896--3901.

\bibitem{camacho2002a}
D.~L. Marruedo, T.~Alamo, and E.~F. Camacho, ``Input-to-state stable mpc for
  constrained discrete-time nonlinear systems with bounded additive
  uncertainties,'' in \emph{Proceedings of the 41st IEEE Conference on Decision
  and Control}, 2002, pp. 4619--4624.

\bibitem{khalil}
H.~K. Khalil, \emph{Nonlinear Systems}, 3rd~ed., Prentice Hall, 2001.

\end{thebibliography}
\end{document}